\documentclass[12pt]{amsart}

\usepackage{fullpage,amsmath, amssymb, xypic,xr}

\newcommand{\ubM}{M_{\leq 1}}
\newcommand{\ub}[1]{(#1)_{\leq 1}}

\newcommand{\bC}{\mathbf C} % the set of new constant symbols
\newcommand{\Lang}{\mathcal L} % Language
\DeclareMathOperator{\Th}{Th} % Theory
\DeclareMathOperator{\Hom}{Hom} % Homomorphisms

\newcommand{\bT}{\mathbf T}

\newcommand{\bbN}{{\mathbb N}}
\newcommand{\bbQ}{{\mathbb Q}}
\newcommand{\bbR}{{\mathbb R}}

\newcommand{\bbC}{\mathbb C}

\DeclareMathOperator{\id}{id}

\newcommand{\CC}{C}

\newcommand{\cF}{\mathcal F}

\newcommand{\cB}{\mathcal B}

\newcommand{\e}{\varepsilon}

\newtheorem{thm}{Theorem}[section]
\newtheorem{theorem}[thm]{Theorem}
\newtheorem{corollary}[thm]{Corollary}

\newtheorem{question}[thm]{Question}

\newtheorem{lemma}[thm]{Lemma}
\newtheorem{proposition}[thm]{Proposition}
\newtheorem{prop}[thm]{Proposition}

\theoremstyle{definition}

\newtheorem{notat}[thm]{Notation}
\newtheorem{definition}[thm]{Definition}

\newtheorem{example}[thm]{Example}
\theoremstyle{remark}
\newtheorem{remark}[thm]{Remark}

\DeclareMathOperator{\Fin}{Fin}

\newcounter{my_enumerate_counter}
\newcommand{\pushcounter}{\setcounter{my_enumerate_counter}{\value{enumi}}}
\newcommand{\popcounter}{\setcounter{enumi}{\value{my_enumerate_counter}}}

 \DeclareMathOperator{\tr}{tr}

\DeclareMathOperator{\Tr}{Tr}

\newcommand{\bfT}{{\mathbf T}}

\newcommand{\cU}{{\mathcal U}}
\newcommand{\cV}{\mathcal V}

\newcommand{\bbZ}{\mathbb Z}

\newcommand{\bfa}{\mathbf a}
\newcommand{\bfb}{\mathbf b}

\newcommand{\fc}{\mathfrak c}

\def\CC{{\mathcal{C}}}
\def\CL{{\mathcal L}}
\def\CP{{\mathcal P}}
\def\CM{{\mathcal M}}
\def\CN{{\mathcal N}}
\def\CU{{\mathcal U}}
\def\CV{{\mathcal V}}
\def\CA{{\mathcal A}}
\def\CB{{\mathcal B}}
\def\CD{{\mathcal D}}
\def\CS{{\mathcal S}}

\title{Model theory of operator algebras II: Model theory}

\author{Ilijas Farah}

\address{Department of Mathematics and Statistics\\
York University\\
4700 Keele Street\\
North York, Ontario\\ Canada, M3J
1P3\\
and Matematicki Institut, Kneza Mihaila 34, Belgrade, Serbia}

\email{ifarah@mathstat.yorku.ca}
\urladdr{http://www.math.yorku.ca/$\sim$ifarah}

\author{Bradd Hart}
\address{Dept. of Mathematics and Statistics\\
McMaster University\\ 1280 Main Street\\ West Hamilton, Ontario\\
Canada L8S 4K1}

\email{hartb@mcmaster.ca}

\urladdr{http://www.math.mcmaster.ca/$\sim$bradd/}
\thanks{The first two authors are partially supported by NSERC}

\author{David Sherman}

\address{Department of Mathematics\\
University of Virginia\\
P. O. Box 400137\\
Charlottesville, VA 22904-4137}
\email{dsherman@virginia.edu}

\urladdr{http://people.virginia.edu/$\sim$des5e/}
\date{\today}

\begin{document}

\begin{abstract}
We  introduce a version of logic for metric structures suitable for
applications to C*-algebras and tracial von Neumann algebras.
We also prove a purely
model-theoretic result to the effect that the theory of a separable metric
structure is stable if and only if all of its ultrapowers associated
with nonprincipal ultrafilters on $\bbN$ are isomorphic even when the
Continuum Hypothesis fails.

\end{abstract}

\maketitle

\section{Introduction}

The present paper is a companion to \cite{FaHaSh:Model1}.
The latter paper was written in a way that completely suppressed
explicit use of logic, and model theory in particular, in order to
be more accessible to  operator algebraists.
Among other results, we will prove a metatheorem (Theorem~\ref{T0})
that explains results of \cite{FaHaSh:Model1}, as well as the word `stability'
in its title.

We will study operator algebras using a slightly modified version of
the \textit{model theory for metric structures}.  This is a logical
framework whose semantics are well-suited for the approximative
conditions of analysis; as a consequence it plays the same role for
analytic ultrapowers as first order model theory plays for classical
(set theoretic) ultrapowers.  We show that the continuum hypothesis
(CH) implies that all ultrapowers of a separable metric structure
are isomorphic, but under the negation of CH this happens if and
only if its theory is \textit{stable}.  Stability is defined in
logical terms (the space of $\varphi$-types over a separable model
is itself separable with a suitable topology), but it can be characterized as follows: a theory is \textit{not} stable if and
only if one can define arbitrarily long finite ``uniformly
well-separated" totally ordered sets in any model, a condition
called the \textit{order property}.  Provided that the class of
models under consideration (e.g., II$_1$ factors) is defined by a theory -- not always obvious or even true -- this brings the main question back into the arena of
operator algebras. To deduce the existence of nonisomorphic
ultrapowers under the negation of CH, one needs to establish the order
property by defining appropriate posets.   We proved in \cite{FaHaSh:Model1} that all infinite-dimensional C*-algebras and II$_1$ factors have the order property, while tracial von Neumann
algebras of type I do not.  In a sequel paper we will use the logic developed here to obtain new results about isomorphisms and embeddings between $\text{II}_1$ factors and their ultrapowers.
%REWRITTEN AND AUGMENTED ONCE
%THE REST OF THE PAPER IS DONE.  PROBABLY EARLIER WE SHOULD SAY THAT
%THIS PART OF OUR APPROACH WAS PARTIALLY CARRIED OUT BY GE-HADWIN
%(P.312-313). THAT ISN'T MENTIONED.

%All of the main results of this paper (Theorem \ref{T.type-II-1}, Proposition \ref{P.A1} and Theorem \ref{T.C*}) follow from the fact that the continuous theory (definitions in \S \ref{S.Theory} and \S\ref{S.Elementary}) of the separable object in question is unstable and a general result (Theorem \ref{T0}) which we prove in \S \ref{S.Model-theoretic}.  In fact in Theorem \ref{T0}  we also characterize metric structures with the property that all of their ultrapowers associated with nonprincipal ultrafilters on $\bbN$ are isomorphic regardless of whether the Continuum Hypothesis holds.

We now review some facts and terminology for operator algebraic ultrapowers
that we will use throughout the paper; this is reproduced for
convenience from \cite{FaHaSh:Model1}.

A von Neumann algebra $M$ is \textit{tracial} if it is equipped with a faithful normal tracial state~$\mbox{tr}$.  A finite factor has a unique tracial state which is automatically normal.  The metric induced by the $\ell^2$-norm, $\|a\|_2=\sqrt{\mbox{tr}(a^*a)}$, is not complete on $M$, but it is complete on the (operator norm) unit ball of $M$.  The completion of  $M$ with respect to this metric is
 isomorphic to a Hilbert space (see, e.g.,
\cite{Black:Operator}~or~\cite{Jon:von}).

The algebra of all sequences in $M$ bounded in the operator norm is denoted by $\ell^\infty(M)$.
If~$\cU$ is an ultrafilter on $\bbN$ then
\[
\textstyle c_{\cU}=\{\vec a\in \ell^\infty(M): \lim_{i\to \cU} \|a_i\|_2=0\}
\]
is a norm-closed two-sided ideal in
$\ell^\infty(M)$, and the \emph{tracial ultrapower} $M^{\cU}$ (also
denoted by $\prod_{\cU} M$) is defined to be the
quotient  $\ell^\infty(M)/c_{\cU}$.  It is well-known that $M^{\cU}$ is tracial, and a factor if and only if $M$ is---see, e.g., \cite{Black:Operator} or
\cite{Tak:TheoryIII}; this also follows from axiomatizability (\S\ref{S.Ax.II-1}) and \L o\' s's theorem (Proposition~\ref{P.Los} and the remark afterwards).

Elements of $M^{\cU}$ will either be denoted by boldface Roman letters such as
$\bfa$ or represented by sequences in $\ell^\infty(M)$.
Identifying a tracial von Neumann algebra $M$ with its diagonal image in  $M^{\cU}$, we will also work with the \emph{relative commutant} of $M$ in its ultrapower,
\[
M'\cap M^{\cU}=\{\bfb: (\forall a\in M) a\bfb=\bfb a\}.
\]

Tracial ultrapowers were first constructed in the 1950s and became standard tools after the groundbreaking papers of McDuff (\cite{McDuff:Central}) and Connes (\cite{Connes:Class}).  Roughly speaking, the properties of an ultrapower are the approximate properties of the initial object; see \cite{She:Notes} for a recent discussion of this.

In defining ultrapowers for C*-algebras (resp. groups with bi-invariant metric), $c_{\cU}$ is taken to be the sequences that converge to zero in the operator norm (resp. converge to the identity in the metric (\cite{Pe:Hyperlinear})).  All these constructions are special cases of the ultrapower/ultraproduct of metric structures (see \S\ref{S.Logic}, also \cite{BYBHU} or \cite{GeHa}).

 %The logic of metric structures and its model theory are used only in \S\ref{S.Model-theoretic}. In \S\ref{S.Syntax}  we shall outline a variation on the logic for metric structures  that is more convenient  for study of II$_1$ factors, C*-algebras, and other structures based on Banach spaces and  show that it is equivalent to the logic from \cite{BYBHU}.

\section{Logic}
\label{S.Logic}
\label{S.Syntax}
The purpose of this section is to introduce a logic which has some features geared to the treatment of $C^*$-algebras and von Neumann algebras.  In a treatment of such structures in bounded continuous logic (see \cite{BYU:ContStab}), it is typical to consider different sorts of balls of increasing radius.  The logic presented here is entirely equivalent to that formulation but allows us to introduce function symbols like $+$ and $\cdot$ without treating them as infinitely many different functions mapping between sorts.  This distinction is somewhat cosmetic but the treatment of terms in this logic highlights an issue that is common to both this logic and the multi-sorted version.  Details are given below but to make clear what is at stake, suppose we are considering a normed linear space and we wish to assert that the unit ball is convex.  The operation $+$ when restricted to the unit ball would most naturally map to the ball of radius 2.  Scalar multiplication by 1/2 maps the ball of radius 2 into the unit ball and so a natural way to set things up would be to have the term $(x+y)/2$ send the unit ball to itself and so the syntax guarantees that the unit ball is convex.  If on the other hand, the scalar 1/2 on the ball of radius 2 was said to have range that same ball (a logical possibility), then $(x+y)/2$ syntactically would only map the unit ball to the ball of radius 2 and we would need to have an axiom that said that this term in fact has range in the unit ball.  Issues of the axiomatizability of the classes of structures we are dealing with are bound up with the choice of range of terms in our language and are highlighted below.

%The axiomatizability of II$_1$ factors (\S\ref{S.Ax.II-1}) was proved by Ben Yaacov, Henson, Junge and Raynaud (unpublished).

\subsection{Language}  \label{SS.Language}

A language consists of
\begin{itemize}
\item Sorts, $\CS$, and for each sort $S \in \CS$, a set
of domains $\CD_S$ meant to be domains of quantification, and a
privileged relation symbol $d_S$ intended to be a metric. Each sort
comes with a distinct set of variables.
\item Sorted functions, $f:S_1\times\ldots\times S_n \rightarrow
S$ together with, for every choice of domains $D_i \in \CD_{S_i}$,  a
$D_{\bar D}^f \in \CD_S$ and for each $i$, a uniform continuity
modulus $\delta^{\bar D,\tau}_i$, i.e., a real-valued function on $\bbR$,
where $\bar D =\langle D_1,\ldots, D_n\rangle$.

\item Sorted relations $R$ on $S_1\times\ldots\times S_n$ such that
for every choice of domains $\bar D$ as above, there is a number
$N^R_{\bar D}$ as well as uniform continuity moduli dependent on $i$
and $\bar D$.

\item Terms are formed by the usual composition of function symbols and variables.  They inherit codomains and series of uniform continuity moduli from this composition.
\end{itemize}
%In the above the term `domain' is used in the model-theoretic sense. In particular domains $D_i$ are not assumed to be open sets.

\subsection{Structures} A structure $\CM$ assigns to each sort $S
\in \CS$, $M(S)$, a metric space where $d_S$ is interpreted as the
metric.  For each $D \in \CD_S$, $M(D)$ is a  subset of $M(S)$
complete with respect to $d_S$. The collection $\{M(D) : D \in
\CD_S\}$ covers $M(S)$.

Terms $\tau$ are interpreted as functions on a structure in the
usual manner.  If $\tau^M$ is the interpretation of $\tau$ and $\bar D$ is
a choice of domains from the relevant sorts then $\tau^M:M(\bar D)
\rightarrow M(D^\tau_{\bar D})$ and $\tau^M$ is uniformly continuous as
specified by the $\delta^{\bar D,\tau}$'s when restricted to $M(\bar
D)$.  This means for instance that for every $\epsilon > 0$, if $a,b \in M(D_1)$ and $c_i \in
M(D_i)$ for $i = 2,\ldots,n$ then $d_S(a,b) <  \delta^{\bar
D,\tau}_1(\epsilon)$ implies $d_{S'}(\tau(a,\bar c),\tau(b,\bar c)) \leq
\epsilon$, where $S$ is the sort associated to $D_1$ and $S'$ is the
sort associated with the range of $\tau$.

Sorted relations are maps $R^M:S_1\times\ldots\times S_n \rightarrow
\bbR$. They are handled similarly to sorted functions; uniform continuity
is as above when restricted to the appropriate domains and a
relation $R$ is bounded in absolute value by $N^R_{\bar D}$ when
restricted to $M(\bar D)$.

\subsection{Examples}
\subsubsection{C*-algebras} \label{Ex.C*}
We will think of a C*-algebra $A$
as a one-sorted
structure with sort $U$ for the algebra itself.  The
domains for $U$ are $D_n$ for every $n \in \bbN$ and are interpreted as all $x \in A$ with
$\|x\| \leq n$.  The metric on $U$ is
\[
d_U(a,b)=\|a-b\|.
\]

The functions in the language will be:
\begin{itemize}
\item The constant 0 which will be in $D_1$.  Note it is a requirement
of the language to say this.
\item For every $\lambda \in \bbC$ a unary function symbol also denoted $\lambda$ to be interpreted as scalar multiplication.
For simplicity we shall write $\lambda x$ instead of $\lambda(x)$.
\item A unary function symbol  $*$ for involution.
\item Binary function symbols  $+$ and $\cdot$.
\end{itemize}
Prescribing the uniform continuity moduli is straightforward.
%One subtlety here is that given a term $\tau$, a *-polynomial in finitely many variables, we want to fix optimal bounds on the range of $\tau$; that is, if 
%$\CC$ is the class of C*-algebras and given a sequence of domains $\bar D$, let $D^\tau_{\bar D} = D_N$ where N is
%$$\min\{ n : \mbox{for all } A \in \CC, \bar a \in \bar D (A), \|\tau(\bar a)\| \leq n\}$$ 

If $A$ is a C*-algebra then there is a model, $\CM(A)$,
in $\CL_{C^*}$ associated to it which is essentially $A$ itself
endowed with the domains $D_n$ interpreted as the operator norm $n$-ball.

\subsubsection{Tracial von
Neumann algebras}
\label{Ex.II-1}
Tracial
von Neumann algebras will be treated as a one-sorted
structure with domains $D_n$ which as in the example of C*-algebras will be interpreted as the operator norm $n$-ball. The metric $d$ will
be the metric arising from the $\ell^2$ norm coming from the trace.   

The functions in the language are, in addition to  functions from \S\ref{Ex.C*},
\begin{itemize}
\item The constant 1 in $D_1$.
\item Two unary relation symbols $\tr^r$ and $\tr^i$ for the real and imaginary parts of the
trace function.  We will often just write $\tr$ and assume that the expression can be decomposed into the real and imaginary parts.
\end{itemize}
Again, this describes a language $\CL_{\Tr}$ once we add the
requirements about bounds on the range and uniform continuity. 
%For terms in this language, we fix optimal bounds as in the case of C*-algebras.  To be precise, if $\CN$ is the class of tracial von Neumann algebras then 
%given a term $\tau$ and a sequence of domains $\bar D$, let $D^\tau_{\bar D} = D_N$ where N is
%$$\min\{ n : \mbox{for all } A \in \CN, \bar a \in \bar D (A), \|\tau(\bar a)\| \leq n\}.$$

If $N$ is a
tracial von Neumann algebra then there is a model, $\CM(N)$, in
$\CL_{\Tr}$ associated to it which is essentially $N$ itself with the domains interpreted as above.

\begin{remark}
We emphasize that the operator norm is not a part of the language,
and that it is not even a definable relation. Note that all relations
are required to be uniformly continuous functions, and $\|\cdot\|$ is not
uniformly continuous with respect to $\|\cdot\|_2$.
\end{remark}

\subsubsection{Unitary groups}
The  syntax for logic of   unitary groups is simpler than that of tracial von Neumann algebras or C*-algebras. In this case the metric is bounded and therefore we can have one domain are equal to the universe $U$.
%Terms of the language are group words.
We have function symbols for the identity, inverse and the group
operation. Since in this case our logic reduces to the
standard logic of metric structures as introduced in \cite{BYBHU}
we omit the straightforward details and continue this
practice of suppressing the details for unitary groups throughout
this section.

\subsection{Syntax}\label{SS.Syntax}
\begin{itemize}
\item Formulas:
\begin{itemize}
\item If $R$ is a relation and $\tau_1,\ldots,\tau_n$ are terms then
$R(\tau_1,\ldots,\tau_n)$ is a basic formula.
\item If $f:\bbR^n \rightarrow \bbR$ is continuous and
$\varphi_1,\ldots,\varphi_n$ are formulas, then
$f(\varphi_1,\ldots,\varphi_n)$ is a formula.
\item If $D\in \CD_S$ and $\varphi$ is a formula then both $\sup_{x
\in D} \varphi$ and $\inf_{x \in D} \varphi$ are formulas.
\end{itemize}

\item Formulas are interpreted in the obvious manner in structures.
 The boundedness of relations when restricted to domains is essential to guarantee that
the sups and infs exist when interpreted. For a fixed formula
$\varphi$ and real number $r$, the expressions\\ $\varphi \leq r$ and
$r \leq \varphi$ are called {\em conditions} and are either true or false
in a given interpretation in a structure.
\end{itemize}

\subsubsection{The expanded language}\label{S.Expanded.Language}
In the above definition it was taken for granted that we have an infinite supply of distinct variables
appearing in terms. In~\S\ref{S.Types} below we shall need to introduce a set of new
constant symbols~$\bC$. Each $c\in \bC$ is assigned a sort $S(c)$ and a domain.
In the \emph{expanded language} $\Lang_{\bC}$ both variables and
constant symbols from $\bC$ appear in terms. Formulas and sentences
in $\Lang_{\bC}$ are defined as above. Note that, since  the
elements of $\bC$ are not variables, we do not allow quantification
over them.

\subsection{Theories and elementary equivalence} \label{S.Theory}
A sentence is a formula with no free variables.
 If $\varphi$ is a sentence and $\CM$ is a structure
then the result of interpreting $\varphi$ in $\CM$ is a real number,
$\varphi^\CM$.  The function which assigns these numbers to
sentences is the \emph{theory of $\CM$}, denoted by $\Th(\CM)$. Because we allow all continuous functions as connectives, in particular the functions $|x - \lambda|$, the
theory of a model $\CM$ is uniquely determined by its zero-set,
$\{\varphi: \varphi^{\CM}=0\}$. We shall therefore adopt the
convention that a set of sentences $T$ is a theory and say that
\emph{$\CM$ is a model of $T$}, $\CM\models T$,  if
$\varphi^{\CM}=0$ for all $\varphi\in T$.

The following is proved by induction on the complexity of the
definition of $\psi$.
% in case of both  C*-algebras and tracial von
%Neumann algebras  (see \cite{BYBHU}).

\begin{lemma} \label{L.formula} Suppose
 $\CM$ is a model and $\psi(\bar x)$ is a formula, possibly with parameters
 from $M$. For every choice of  $\bar D$ sequence of domains consistent with the sorts of the variables, $\psi^{\CM}$ is a uniformly
continuous function on  $M(\bar D)$ into a compact subset of $\bbR$.

If $\Theta\colon \CM\to \CN$ is an isomorphism
then $\psi^{\CM}=\psi^{\CN}\circ \Theta$. \qed
\end{lemma}

% \label{S.Elementary}
%In this section we
%talk about models. However, all the discussion carries through for any axiomatizable class of metric
%structures and in particular for C*-algebras and for II$_1$ factors
%by  Proposition~\ref{P.C*-completeness} and Proposition~\ref{P.vNA-completeness}.
%Moreover, the notions defined below naturally lift to the underlying algebras.
%We shall therefore, for example,
% talk about a II$_1$ factor being an elementary subalgebra of another II$_1$ factor.

Two models  $\CM$ and $\CN$ are \emph{elementarily equivalent}  if
$\Th(\CM)=\Th(\CN)$. %and a map $\Theta\colon \CM\to \CN$ is an
%\emph{elementary embedding} if it is an elementary embedding of
%$\CM(\CM)$ into $\CM(\CN)$.
A map $\Theta\colon \CM\to \CN$ is an
\emph{elementary embedding} if
 for all formulas $\psi$ with parameters in $M$, we have $\psi^{\CM}=\psi^{\CN}\circ \Theta$.

If $\CM$ is a submodel  of $\CN$ and the identity map from $\CM$ into $\CN$
is elementary then we say that
$\CM$ is an \emph{elementary submodel} of $\CN$.
It is not difficult to see that every elementary embedding is an
isomorphism onto its image,\footnote{an isomorphism in the
appropriate category; in case of operator algebras this is
interpreted as `*-isomorphism'} but not vice versa.

\section{Axiomatizability}

\begin{definition}
A category $\CC$  is \textbf{axiomatizable} if there is  a language
$\CL$ (as above), theory $T$ in  $\CL$,  and a collection of
conditions $\Sigma$ such that $\CC$ is equivalent to the category of
models of $T$ with morphisms given by maps that preserve $\Sigma$.
\end{definition}

%\begin{example} Let $T$ be a theory in a language $\CL$.  If $\Sigma$
%is
%\begin{enumerate}
%\item  all conditions on all formulas, then the morphisms preserving
%$\Sigma$ are just elementary maps;
%\item all conditions of the form $\varphi \leq r$ and $r \leq
%\varphi$ for basic formulas $\varphi$, then the morphisms are just
%embeddings;
%\item all conditions of the form $\varphi \leq r$ for basic formulas
%$\varphi$, then the morphisms are homomorphisms.
%\end{enumerate}
%In the first two examples the maps will be isometries, but not in
%the third.
%\end{example}

The reason for being a little fussy about axiomatizability is that
in the cases we wish to consider, the models have more (albeit
artificial) `structure' than the underlying algebra (cf.
\S\ref{Ex.C*} and \S\ref{Ex.II-1}). The language of the model will
contain operation symbols for all the algebra operations (such as
$+,\cdot$ and~$*$) and possibly some distinguished constant symbols
(such as the unit) and predicates (e.g.,
 a distinguished state on a C*-algebra). It will also contain domains that
 are not part of the algebra's structure.

 In particular then, when we say that we have axiomatized a class of
 algebras $\CC$, we will mean that there is a first order continuous
 theory $T$ and specification of morphisms such that
 \begin{itemize}
 \item for any $A
 \in \CC$, there is a model $M(A)$ of $T$ determined up to
 isomorphism;
 \item for any model $M$ of $T$ there is $A \in \CC$ such that
 $M$ is isomorphic to $M(A)$;
 \item if $A,B \in \CC$ then there is a
 bijection between $\Hom(A,B)$ and $\Hom(M(A) ,M(B))$.
 \end{itemize}

Proving that a category is axiomatizable frequently involves somewhat tedious
syntactical considerations. However, once this is proved we can apply a variety of model-theoretic
tools to study this category. In particular, we can immediately conclude that
 the category is closed under taking ultraproducts---a nontrivial theorem in the
 case of tracial von Neumann algebras.
{}From here it also follows that  some natural categories of operator algebras are not axiomatizable
 (see Proposition~\ref{P.UHF}).

\subsection{Axioms for C*-algebras} We continue the discussion of model theory
of C*-algebras  started in \S\ref{Ex.C*}.
First we introduce two notational shortcuts.
If one wants to write
down axioms to express that $\tau = \sigma$ for terms $\tau$ and $\sigma$
then one can write
\[ \varphi_{\bar D} := \sup_{\bar a \in \bar D} d_U(\tau(\bar a),\sigma(\bar a)) \]
where $\bar D$ ranges over all possible choices of domains. Note
that this is typically an infinite set of axioms.  Remember that for
a model to satisfy $\varphi_{\bar D}$, this sentence would evaluate
to 0 in that model.  If this sentence evaluates to 0 for all choices
of $\bar D$ then clearly $\tau = \sigma$ in that model.

If one wants to write
down axioms to express that $\varphi\geq  \psi$ for formulas $\varphi$ and $\psi$
then one can write
\[
 \sup_{\bar a \in \bar D} \max(0,(\psi(\bar a)-\varphi(\bar a)))
\]
where $\bar D$ ranges over all possible choices of domains.  Again,
we will get the required inequality if all these sentences evaluate
to 0 in a model.

Using the above conventions, we are taking the universal closures
of the following formulas, where $x,y,z,a,b$, range over the algebra
and $\lambda,\mu$ range over the complex numbers.

Here are some sentences that evaluate to
zero in a C*-algebra $A$.  The first item guarantees that we have a $\bbC$-vector space.
\begin{enumerate}
\popcounter
\item\label{Ax.associativity}
 $x + (y + z) = (x + y) + z$, $x + 0 = x$, $x + (-x) = 0$ (where $-x$ is the scalar $-1$ acting on $x$),
  $x+y = y+x$,
 $\lambda(\mu x) = (\lambda\mu)x$,
 $\lambda(x+y) = \lambda x + \lambda y$,
 $(\lambda + \mu) x = \lambda x + \mu x$.
\pushcounter
\end{enumerate}

\begin{enumerate}
\popcounter
\item  $1x = x$,
 $x(yz) = (xy)z$,
 $\lambda(xy) = (\lambda x)y = x(\lambda y)$,
 $x(y+z) = xy + xz$; now we have a $\bbC$-algebra.
\item $(x^*)^* = x$, $(x+y)^* = x^* + y^*$, $(\lambda x)^* = \bar \lambda x^*$.
\item \label{Ax.product.*}
$(xy)^* = y^*x^*$.
\item \label{Ax.dU}
$d_U(x,y) = d_U(x-y,0)$; we will write $\|x\|$ for
$d_U(x,0)$.
\pushcounter
\end{enumerate}

\begin{enumerate}
\popcounter
\item $\|xy\|\leq \|x\|\|y\|$.
\item $\|\lambda x\|=|\lambda|\|x\|$.
\item\label{Ax.C*-equality}
 (C*-equality) $\|xx^*\|=\|x\|^2$.
%\pushcounter
%\end{enumerate}
%The last two axioms assure us that each $D_n$ is interpreted as the $n$-ball (and can therefore
%be safely ignored once Proposition~\ref{P.C*-completeness} is proved).
%\begin{enumerate}
%\popcounter
%\item \label{Ax.D-1} For every $k, n \in \bbN$ fix a non-negative function $f:\bbR \rightarrow \bbR$ such that $f$ is $0$ on the interval $[0,1/k]$ and positive elsewhere.  The axiom is then
%$$
%\sup_{a\in D_n} \min \{  \max\{ f(\lambda - \|a\|),\inf_{b\in D_1}\|b-\frac{1}{\lambda} a\|\} : \lambda = j/k, 1 \leq j \leq kn\}
%$$
\item \label{Ax.D1} $\sup_{a \in D_1} \|a\|\leq 1$.
\pushcounter
\end{enumerate}

One issue here is that these axioms are too weak to guarantee that $D_1$ is the operator norm unit ball.  To get around this we expand the language of C$^*$-algebras to include a function symbol $\tau_p$ for every *-polynomial $p$ in one variable.  The symbol $\tau_p$ will have the same uniform continuity modulus as $p$.  In order to determine the proper codomains, for every $n$, let $m$ be the least integer greater than or equal to $\max\{ \| p(a) \| : a \in M, M \in \CC \mbox{ and } \|a\| \leq n \}$ where $\CC$ is the class of C$^*$-algebras.  We will require $\tau_p:D_n \rightarrow D_m$ and we will add the universally quantified axioms 
\begin{enumerate}
\popcounter
\item\label{extra-ax} $\tau_p(x) = p(x)$ 
\pushcounter
\end{enumerate}
for all polynomials $p$.  This will force the polynomial $p$ to behave well with respect to where its range lands.  To see the effect of these axioms, we do a small calculation.

Suppose that $\CM$ is a structure that satisfies axioms 1 through \ref{Ax.D1} above.  Suppose $a \in M$, $\|a\| \leq 1$ and $a \in D_n(\CM)$.   Define
\[ t_n(x) = \left\{ \begin{array}{ll}
			1 & 0 \leq x \leq 1 \\
			\frac{1}{\sqrt{x}} & 1 < x \leq n
			\end{array} \right . \]
and consider $f(u) = ut_n(u^*u)$.  If we want to compute the norm of $f(u)$ for $\|u\| \leq n$, we see that
$\|f(u)\|^2 = \|t_n(u^*u)u^*ut_n(u^*u)\| = \|g(u^*u)\|$ where $g(x) = xt_n^2(x)$.  Since 
\[ g(x) = \left\{ \begin{array}{ll}
			x & 0 \leq x \leq 1 \\
			1 & 1 < x \leq n
			\end{array} \right . \]
we obtain that the norm of $f(u)$ is at most 1 when $\|u\| \leq n$.

Now fix polynomials $p_k(x)$ which tend to $t_n(x)$ from below on the interval $[0,n]$. By doing a calculation similar to the one above, the *-polynomial $q_k = up_k(u^*u)$ sends operators of norm $\leq n$ to operators of norm $\leq 1$.  This means that $\tau_{q_k}$ sends elements of $D_n$  to elements of $D_1$ by the specification of our language for C$^*$-algebras. Moreover, $ap_k(a^*a)$ tends to $a$ as $k$ tends to infinity.  Since $D_1(\CM)$ is complete, we obtain that $a \in D_1(\CM)$.   

\begin{proposition}
\label{P.C*-completeness}
The class of C*-algebras is axiomatizable by
theory  $\bfT_{C^*}$ consisting of axioms \eqref{Ax.associativity}--\eqref{extra-ax}. 
\end{proposition}

\begin{proof}
It is clear that for a C*-algebra $A$ the model $\CM(A)$ as defined
in \S\ref{Ex.C*} satisfies $\bfT_{C^*}$. Conversely, if a model
$\CM$  of $\CL_{C^*}$ satisfies $\bfT_{C^*}$ then the algebra
$A_\CM$ obtained from $\CM$ by forgetting the domains is a C*-algebra by Gel'fand-Naimark.

To see that this provides an equivalence of categories, we only need
to show that $M(A_\CM) \cong \CM$. To see this, we must show that the domains on $\CM$ are determined by $A_\CM$.  Since multiplication by a scalar $r$ provides a bijection between the operator norm unit ball and the ball of radius $r$, it suffices to show that the operator norm unit ball and those elements of $D_1(\CM)$ coincide.  By axiom \ref{Ax.D1}, we have that the latter is contained in the former.  The other direction is just the calculation we did immediately before the Proposition.\end{proof}

\subsection{Axioms for tracial von Neumann algebras}\label{S.Ax.II-1}
We continue our discussion of model theory of tracial von Neumann
algebras from \S\ref{Ex.II-1}. Axioms for tracial von Neumann
algebras and II$_1$ factors appear in the context of bounded
continuous logic in \cite{BYHJR}; those axioms are restricted to axiomatizing the norm one unit ball.  
We feel in this context axiomatizing von Neumann algebras in the logic described in the previous section
makes the axioms more natural. Here are some sentences that
evaluate to zero in a tracial von Neumann algebra $N$:
\begin{enumerate}
\popcounter
\item\label{Ax.vNA.1}
 All axioms \eqref{Ax.associativity}--\eqref{Ax.dU}  plus $1x = x = x1$ for the constant 1 of $N$. In case of \eqref{Ax.dU}
we will write $\|x\|_2$ for
$d_U(x,0)$.
\item $\tr(x+y) = \tr(x) + \tr(y)$
\item \label{Ax.trace.x*}
$\tr(x^*) = \overline{\tr(x)}$,  $\tr(\lambda x) = \lambda \tr(x)$,  $\tr(xy) = \tr(yx)$,  $\tr(1) = 1$,
\item $\tr(x^*x) = \|x\|_2^2$.
\pushcounter
\end{enumerate}
Any model of these axioms will be a tracial *-algebra. %these axioms are enough to guarantee that the structure on $N$ is a Hilbert space with inner product given by $\langle x \mid y \rangle= \tr(x^*y)$.
The remaining axiom will guarantee that the relationship between the domains and the 2-norm is correct.

\begin{enumerate}
\popcounter
\item \label{Ax.vNA.critical}
For every $n, m \in \bbN$,
\[ \sup_{a \in D_n} \sup_{x \in D_m} \max\{ 0, \| ax\|_2 - n\|x\|_2\} \]
%\item \label{Ax.vNA.polar}
%\label{Ax.vNA.last}  (Polar decomposition)
%For every $n\in \bbN$,
%\[
%\sup_{a\in D_n}
%\inf_{b\in D_1}
%\inf_{c\in D_n}
%(\|a-bc^*c\|_2+\|b^*b-1\|_2).
%\]
%This axiom states that $a=bd$ is the polar decomposition of $a$, with $b\in D_1$ being unitary and
%$d$ being positive.
\pushcounter
\end{enumerate}

In addition to these axioms, we also introduce terms $\tau_p$ for all unary *-polynomials $p$ as discussed above for C$^*$-algebras.

\begin{prop}
\label{P.vNA-completeness}
The class of tracial von Neumann algebras is axiomatizable
by theory  $\bfT_{\Tr}$
consisting of axioms  \eqref{extra-ax}--\eqref{Ax.vNA.critical}.
\end{prop}

\begin{proof}  It is clear that for a tracial von Neumann algebra $N$ the model $\CM(N)$
as defined in \S\ref{Ex.II-1} satisfies $\bfT_{\Tr}$.
 Assume $\CM$ satisfies $\bfT_{\Tr}$.
To see that in the sort $U$ we have a tracial von Neumann
algebra suppose $A$ is the underlying set for $U$ in $\CM$. Then  $A$ is a
complex pre-Hilbert space with inner product given by $\tr(y^*x)$.    Left multiplication
by $a \in A$ is a linear operator on $A$ and axiom \eqref{Ax.vNA.critical}
 guarantees that $a$ is bounded.
The operation $*$ is the adjoint because for all $x$ and $y$ we have
$\langle ax,y \rangle = \tr(y^*ax) =\tr((a^*y)^* x)= \langle x, a^*y \rangle$.  Thus $A$ is faithfully represented as a *-algebra of Hilbert space operators.
 We know that $D_n(A)$ is
complete with respect to the 2-norm for all $n$ and the 2-norm induces the strong operator topology on $A$ in this representation; it follows from the Kaplansky density theorem that $A$  is
a tracial von Neumann algebra.

As in the case of C*-algebras above, to show that we have an
equivalence of categories, it will suffice to show that if $\CM$ is
a model of the $\bfT_{\Tr}$ then
$D_1(A)$ is given by the operator norm unit ball on $A$.  Axiom \eqref{Ax.vNA.critical} guarantees that $a \in D_1(A)$ then $\|a\| \leq 1$ and the functional calculus argument from the proof of Proposition \ref{P.C*-completeness} shows $D_1(A)$ equals the operator norm unit ball.
\end{proof}

% axioms for II_1 factors

For $a$ in a tracial von Neumann algebras define the following:
\begin{align*}
\xi(a)&=\sqrt{\|a\|_2^2-\tr^2(a)},\\
\eta(a)&=\sup_{b \in D_1} \|ab-ba\|_2.
\end{align*}
%Let $h\colon \bbR\to \bbR$ be defined as $h(x)=x$ if $x\geq 0$.
Since $\xi$ and $\eta$ are interpretations of terms in the language of
tracial von Neumann algebras,  the following is a sentence of this language.
 \begin{enumerate}
 \popcounter
 \item\label{Ax.vNA.factor}  $\sup_{a \in D_1}  \max\{0,(\xi(a)- \eta(a))\}$.
\pushcounter
\end{enumerate}
 Also consider the axiom
 \begin{enumerate}
 \popcounter
 \item  \label{Ax.vNA.atomless} $\inf_{a \in D_1} (\|aa^*-(aa^*)^2\|_2+|\tr(aa^*)-1/\pi|)$.
\end{enumerate}

\begin{prop} \label{P.II-1-completeness}
\begin{enumerate}
\item \label{P.II-1.completeness.1}
The class of tracial von Neumann factors is axiomatizable
by the theory
consisting of axioms  \eqref{extra-ax}--\eqref{Ax.vNA.factor}.

\item \label{P.II-1.completeness.2}
The class of II$_1$ factors is axiomatizable
by the theory  $\bfT_{II_1}$
consisting of axioms  \eqref{extra-ax}--\eqref{Ax.vNA.atomless}.
\end{enumerate}
\end{prop}

\begin{proof}  For  \eqref{P.II-1.completeness.1},
by Proposition~\ref{P.vNA-completeness}, it suffices to
 prove that if $M$ is a tracial von Neumann algebra then
axiom \eqref{Ax.vNA.factor}
holds in $M$ if and only if $M$ is a factor.
If it is not a factor, let $p$ be a nontrivial central projection.  Then $\xi(p)= \sqrt{\tr(p) - \tr(p)^2} > 0$ but $\eta(p)=0$, therefore \eqref{Ax.vNA.factor}
fails in $M$.  If it is a factor, the inequality $\eta(a) \geq \xi(a)$ follows from \cite[Lemma 4.2]{FaHaSh:Model1}.

For \eqref{P.II-1.completeness.2} we need to
show that axiom \eqref{Ax.vNA.atomless} holds in a tracial factor $M$
if and only if $M$ is type $\text{II}_1$.  When $M$ is type $\text{II}_1$, \eqref{Ax.vNA.atomless} is satisfied by taking $a$ to be a projection of
trace $1/\pi$.  On the other hand, a tracial factor $M$ not of type $\text{II}_1$ is some matrix factor $\mathbb{M}_k$.  If $\mathbb{M}_k$ were to satisfy \eqref{Ax.vNA.atomless}, by compactness of the unit ball there would be $a \in \mathbb{M}_k$ satisfying $\|(aa^*)-(aa^*)^2\|_2=0$ and $|\tr(aa^*) - 1/\pi| = 0$.  Thus $aa^* \in \mathbb{M}_k$ would be a projection of trace $1/\pi$, which is impossible.  (Of course this argument still works if $1/\pi$ is replaced with any irrational number in $(0,1)$.)
%Now assume   holds in $M$.  For any $n$ and $\e>0$ there is a positive element $b\in M$ such that .  It is then a standard exercise in functional calculus to show that $\chi_{[1/2,1]}(b)$ is a projection whose trace is within ?? of $1/n$.  Since $\e$ was arbitrary, $M$ contains projections of trace $1/n$ for every $n$, and $M$ must be type $\text{II}_1$.
\end{proof}

\section{Model-theoretic toolbox}

In the present section we introduce  variants of some of the
standard model-theoretic tools for the logic described in
\S\ref{S.Logic}.

\subsection{Ultraproducts}
Assume $\CM_i$, for $i\in I$, are models of the same language and
$\cU$ is an ultrafilter on $I$. The \emph{ultraproduct} $\prod_{\cU}
\CM_i$ is a model of the same language defined as follows.

In a model $\CM$, we write $S^\CM$ and $D^\CM$ for the interpretations
$S$ and $D$ in $\CM$.  For each sort  $S\in \CS$, let
\[
\textstyle X_S = \{ \bar a \in \prod_{i \in
I} S^{\CM_i} : \mbox{for some } D \in \CD_S, \{i \in I : a_i \in
D^{\CM_i}\} \in \cU\}.
 \]
For $\bar a$ and $\bar b$ in $X_S$,
$d'_S(\bar a,\bar b) = \lim_{i\rightarrow\cU} d_S^{\CM_i}(a_i,b_i)$
defines a pseudo-metric on $X_S$.  Let $S^{\CM'}$ be the quotient
space of $X_S$ with respect to the equivalence $\bar a \sim \bar b$
iff $d'_S(\bar a,\bar b) = 0$ and let $d_S$ be the associated
metric.  For $D \in \CD_S$, let $D^{\CM'}$ be the quotient of $$\{
\bar a \in X_S : \{ i\in I : a_i \in D^{\CM_i} \} \in \cU\}.$$ All
the functions and predicates are interpreted  in the natural way.
Their restrictions to each $\bar D$ are uniformly continuous and
respect the corresponding uniform continuity moduli. If $\CM_i=\CM$
for all $i$ then we call the ultraproduct  an \emph{ultrapower} and
denote it by
 $\CM^{\cU}$.
The `generalized ultraproduct construction' as introduced in
\cite[p. 308--309]{GeHa} reduces to the model-theoretic ultraproduct
in the case of both tracial von Neumann algebras and C*-algebras.

We record a straightforward consequence of the definitions and the
axiomatizability, that the
 functors  corresponding to taking the  ultrapower and defining a model
 commute. The ultrapowers of C*-algebras and tracial von Neumann algebras
 are defined in the usual way.

\begin{proposition}\label{P.ultrapowers}
If  $A$ is a C*-algebra or  a tracial von Neumann algebra and $\cU$
is an ultrafilter then   $\CM(A^{\cU})=\CM(A)^{\cU}$. %In particular,
%An ultrapower of a C*-algebra is a C*-algebra and an ultrapower of a
%tracial von Neumann algebra is a tracial von Neumann algebra.
\end{proposition}

\begin{corollary}\label{C.nonisomorphic-ultrapowers}
A C*-algebra (or a tracial von Neumann algebra) $A$ has
nonisomorphic ultrapowers if and only if the model $\CM(A)$ has
nonisomorphic ultrapowers.
\end{corollary}

\begin{proof}
This is immediate
 by Proposition~\ref{P.C*-completeness},  Proposition~\ref{P.vNA-completeness}
and Proposition~\ref{P.ultrapowers}.
\end{proof}

It is worth remarking that although the proof of Proposition
\ref{P.ultrapowers} is straightforward, this relies on a judicious
choice of domains of quantification.  In general, it is not true
that if one defines domains for a metric structure then the domains
have the intended or standard interpretation in the ultraproduct.
Von Neumann algebras themselves are a case in point.  If we had
defined our domains so that $D_n$ were those operators with
$l_2$-norm less than or equal to $n$ then there would be several
problems.  The most glaring is that these domains are not complete;
even if one persevered to an ultraproduct, the resulting object
would contain unbounded operators.

Ward Henson has pointed out to us that this same problem with
domains manifests itself in pointed ultrametric spaces.  If one defines domains as closed balls of radius $n$ about
the base point, there is no reason to expect that the domains in
an ultraproduct will also be closed balls.  This unwanted phenomenon can be avoided by imposing a geodesic-type condition on the underlying metric; see for instance \cite[Section 1.8]{Carlisle}.

The following is \L o\'s's theorem,
 also known as the Fundamental Theorem of ultraproducts (see \cite[Theorem~5.4]{BYBHU}).
 It is proved by chasing the definitions.

\begin{proposition}\label{P.Los} Let $\CM_i$, $i\in \bbN$, be a sequence of models
of language $\CL$, $\cU$ be an ultrafilter on $\bbN$ and $\CN=
\prod_\cU \CM_i$.
\begin{enumerate}
\item If $\phi$ is an $\CL$-sentence then $\phi^{\CN}=\lim_{i\to \cU} \phi^{\CM_i}$.
\item If $\phi$ is an $\CL$-formula then $\phi^{\CN}(\bfa) =\lim_{i\to \cU} \phi^{\CM_i}(a_i)$,
where $(a_i\colon i\in \bbN)$ is a representing sequence of $\bfa$.
\item The diagonal embedding of a model $\CM$ into $\CM^{\cU}$ is  elementary.
\end{enumerate}
\end{proposition}

%THIS PARAGRAPH AND COROLLARY ARE OUT OF ORDER.
Together with the axiomatizability
(Propositions~\ref{P.C*-completeness} and \ref{P.vNA-completeness})
and Proposition~\ref{P.ultrapowers}, this implies   the  well-known
fact that the ultraproduct of C*-algebras
 (tracial von Neumann algebras, II$_1$ factors, respectively) is a C*-algebra
 (tracial von Neumann algebra, II$_1$ factor, respectively).

%In \cite[Properties~\ref{I-Properties.g}]{FaHaSh:Model1} we assert
In the setting of tracial von Neumann algebras, we have that for any
formula $\phi(x_1,\ldots,x_n)$ with variables from the algebra sort
there is a uniform continuity modulus $\delta$ such that for every
tracial von Neumann algebra $\CM$, $\phi$ defines a function $g$ on
the operator norm unit ball of $\CM$ which is uniformly continuous
with respect to $\delta$ and naturally extends to the operator norm
unit ball of any ultrapower of $\CM$.
%\begin{enumerate}
%\item [(G1)] For every choice of domains $D_1,\ldots,D_n$ consistent with the variables $x_1,\ldots,x_n$ and
%every tracial von Neumann algebra $\CM$, there is a uniform
%continuity modulus $\delta$ such that $\phi$ defines a function $g$
%which is uniformly continuous with respect to  on the $2n$-th power
%of the unit ball of any tracial von Neumann algebra.
%\item [(G2)] For every ultrafilter $\cU$ the function $g$ can be canonically extended to
%the $2n$-th power of the unit ball of the ultrapower
%$(\ubM)^{\cU}=\ub{M^{\cU}}$.
%\end{enumerate}

%In \cite[Properties~\ref{I-Properties.g}]{FaHaSh:Model1}
In \cite{FaHaSh:Model1} we dealt with functions $g$ satisfying the
properties in the previous paragraph and used them to define a linear
ordering showing that some ultrapowers and relative commutants are
nonisomorphic. Using model theory, we can interpret this in a more
general context
 and instead of `tracial von Neumann
algebra' consider $g$ defined with respect to any axiomatizable
class of operator algebras. Clearly, Lemma~\ref{L.formula} and
Proposition~\ref{P.Los} together imply the following, used in the
proof of Theorem~\ref{T0}.

\begin{corollary}\label{C.Properties.g}
If $\psi$ is an $n$-ary formula, then the function $g$ defined
to be the interpretation of $\psi$ on a tracial von Neumann algebra $M$ satisfies the following
\cite[Properties 2.1]{FaHaSh:Model1}:
\begin{enumerate}
\item [(G1)] $g$ defines a uniformly continuous function on the $n$-th power of the unit ball of $M$;
the uniform continuity does not depend on the particular algebra i.e. for every $\epsilon$ there is a $\delta$ independent of the choice of algebra;
\item [(G2)] For every ultrafilter $\cU$ the function $g$ can be canonically extended to
the $n$-th power of the unit ball of the ultrapower $(\ubM)^{\cU}=\ub{M^{\cU}}$. \qed
\end{enumerate}
\end{corollary}

\subsection{Downwards L\"owenheim--Skolem Theorem}
The cardinality of the language and the number of formulas are crude
measures of the L\"owenheim-Skolem cardinal for continuous logic. We
define a topology on formulas relative to a given continuous theory
in order to give a better measure.

Suppose $T$ is a continuous theory in a language $\CL$.  Fix
variables $\bar x = x_1\ldots x_n$ and domains $\bar D = D_1\ldots
D_n$ consistent with the sorts of the $x$'s.  For formulas $\varphi$ and $\psi$ defined on $\bar D$, set
\[
d^T_{\bar D}(\varphi(\bar x),\psi(\bar x)) = \sup\{\sup_{\bar x
\in (\bar D^\CM)^n} |\varphi(\bar x) - \psi(\bar x)| : \CM \models T \}.
\]
Now $d^T_{\bar D}$ is a pseudo-metric; let $\chi(T,\bar D)$ be the
density character of this pseudo-metric on the formulas in the
variables $\bar x$ and define the density character of $\CL$ with
respect to $T$, $\chi(T)$, as $\sum_{\bar D} \chi(T,\bar D)$.

We will say that $\CL$ is \textit{separable} if the density character of
$\CL$ is countable with respect to all $\CL$-theories. Note that the
languages considered in this paper, in particular  $\Lang_{\Tr}$ and
$\Lang_{C^*}$, are separable.

\begin{proposition} \label{P.Separable.language}
Assume $\Lang$ is a separable language. Then for every model $\CM$
of $\Lang$ the set of all interpretations of formulas of $\Lang$ is
separable in the uniform topology.
\end{proposition}

\begin{proof} Since we are allowing all continuous
real functions as propositional connectives (\S\ref{SS.Language})
the set of formulas is not countable. However, a straightforward
argument using  polynomials with rational coefficients and the
Stone--Weierstrass theorem gives a proof.
\end{proof}

The following is a version of the downward L\"owenheim--Skolem
theorem (cf.  \cite[Proposition~7.3]{BYBHU}). Some of its instances
have been rediscovered and applied  in the context of C*-algebras
(see, e.g.,  \cite{Phi:Simple} or the discussion of SI properties
in  \cite{Black:Operator}).  We use the notation $\chi(X)$ to represent
the density character of a set $X$ in some ambient topological
space.

\begin{theorem} Suppose that $\CM$ is a metric structure and $X
\subseteq M$.  Then there is $\CN \prec \CM$ such that $X \subseteq
N$ and $\chi(\CN) \leq \chi(\Th(\CM)) + \chi(X)$.
\end{theorem}

\begin{proof} Fix $\cF$, a dense set of formulas, witnessing $\chi(\Th(\CM))$.
Define two increasing sequences $\langle X_n : n \in \bbN
\rangle$ and $\langle E_n : n \in \bbN \rangle$ of subsets of $M$
inductively so that:
\begin{enumerate}
\item $X_0 = X$;
\item $E_n$ is dense in $X_n$ and $\chi(X_n) = |E_n|$ for all $n \in
\bbN$;
\item $\chi(X_n) \leq \chi(\Th(\CM)) + \chi(X)$; and,
\item for every rational number $r$, formula $\varphi(x,\bar y) \in
\cF$, domain $D$ in the sort of the variable $x$ and $\bar a =
a_1,\ldots,a_k \in E_n$ where $k$ is the  length of $\bar y$, if $\CM \models
\inf_{x \in D} \varphi(x,\bar a) \leq r$ then for every $n>0$ there is $b \in
X_{n+1} \cap D(\CM)$ such that $\CM \models \varphi(b,\bar a)\leq r+(1/n)$.
\end{enumerate}
It is routine to check that $\overline{\cup_{n \in \bbN} X_n}$ is the universe
of an elementary submodel $\CN \prec \CM$ having the correct density character.
\end{proof}

\begin{corollary} \label{P.LST} Assume $\Lang$ is separable.
If $\CM$ is a model of $\Lang$ and $X$ is an infinite  subset of
its universe, then $\CM$ has an elementary submodel whose 
density character is not greater than that of $X$ and whose universe
contains $X$.
\end{corollary}

\subsection{Types} \label{S.Types}
Suppose that $\CM$ is a model in a language $\CL$, $A \subseteq M$
and $\bar x$ is a tuple of free variables thought of as the type
variables.

We follow \cite[Remark 3.13]{BYBHU} and say that a \emph{condition}
over $A$ is an expression of the form $\varphi(\bar x,\bar a) \leq
r$ where $\varphi \in \CL$, $\bar a \in A$ and $r \in \bbR$.  If
$\CN \succ \CM$ and $\bar b \in N$ then $\bar b$ satisfies
$\varphi(\bar x,\bar a) \leq r$ if $\CN$ satisfies $\varphi(\bar
b,\bar a) \leq r$.

Fix a tuple of domains $\bar D$ consistent with $\bar x$, i.e., if
$x_i$ is of sort $S$ then $D_i$ is a domain in $S$. A set of
conditions over $A$ is called a $\bar D$-type over $A$. A $\bar
D$-type is {\em consistent} if for every finite $p_0 \subseteq p$
and $\epsilon > 0$ there is $\bar b \in \bar D(M)$ such that if
$``\varphi(\bar x,\bar a) \leq r" \in p_0$ then $M$ satisfies
$\varphi(\bar b,\bar a) \leq r + \epsilon$. We say that a $\bar
D$-type $p$ over $A$ is {\em realized} in $\CN \succ \CM$ if there
is $\bar a \in \bar D(N)$ such that $\bar a$ satisfies every
condition in $p$.  The following proposition links these two
notions:

\begin{proposition} \label{L.realizing-type} The following are equivalent:
\begin{enumerate}
\item $p$ is consistent.
\item $p$ is realized in some $\CN \succ \CM$.
\item $p$ is realized in an ultrapower of $\CM$.
\end{enumerate}
\end{proposition}

\proof 3) implies 2) and 2) implies 1) are clear.  To see that 1)
implies 3), let $F \subseteq p \times \mathbb{R}_+$ be a finite set, and let $\bar b_F \in
\bar D(M)$ satisfy $\varphi(\bar x,\bar a) \leq r + \epsilon$ for
every $(\varphi(\bar x,\bar a) \leq r,\epsilon) \in F$.  Let $\CU$
be a non-principal ultrafilter over $\CP_{fin}(p \times \mathbb{R}_+)$.  Then $p$
is realized by $(\bar b_F : F \in \CP_{fin}(p \times \mathbb{R}_+))/\CU$ in
$\CM^{\CU}$. \qed

A maximal consistent $\bar D$-type is called \textit{complete}. Let $S^{\bar
D}(A)$ be the set of all complete $\bar D$-types over~$A$.  In fact,
$p$ is a complete $\bar D$-type over $A$ iff $p$ is the set of all
conditions true for some $\bar a \in \bar D(\CN)$ where $\CN \succ
\CM$, by Lemma~\ref{L.realizing-type} and
Proposition~\ref{P.Los}.

\begin{notat} \label{Not.type}
Assume $p$ is a complete type over $A$ and $\phi(x,\bar a)$ is a
formula with parameters $\bar a$ in $A$. Since $p$ is consistent and
maximal, there is the unique real number $r=\sup\{s\in \bbR\colon$
the condition $\phi(x,\bar a)\leq s$ is in $p\}$. In this situation
we shall extend the notation by writing $ \phi(p,\bar a)=r$. We
shall also use expressions such as $|\phi(p,\bar a)-\phi(p,\bar
b)|>\e$.
\end{notat}

We will also often omit the superscript $\bar D$ when it either does
not matter or is implicit.

The set $S^{\bar D}(A)$ carries two topologies: the logic topology and the metric topology.

Fix $\varphi$, $\bar a \in A$ and $r \in \bbR$. A basic closed set
in the logic topology has the form $$\{p \in S^{\bar D}(A) :
\varphi(\bar x,\bar a) \leq r \in p\}$$  The compactness theorem
shows that this topology is compact and it is straightforward that
it is Hausdorff.

We can also put a metric on $S^{\bar D}(A)$ as follows: for $p,q \in
S^{\bar D}(A)$ define
\[
 d(p,q) = \inf\{ d(a,b) : \mbox{there is an }\CN \succ \CM,
\mbox{$a$ realizes $p$ and $b$ realizes $q$}\}.
 \]  The metric
topology is in general finer than the logic topology due to the
uniform continuity of formulas.

\begin{example} Let $\CM$ be a model corresponding to a tracial von Neumann algebra
 or a unital C*-algebra.
\begin{enumerate}
\item
 The \emph{relative commutant type} of $\CM$ is the type over $M$
 consisting of all conditions of the form
 \[
 d([a,x], 0) = 0
 \]
 for $a\in M$.
 \item \label{Ex.2} Another type over $\CM$ consists of all conditions of the form
 \[
 d(a,x) \geq \e
 \]
 for  $a\in M$ and a  fixed $\e>0$.
 \end{enumerate}
 \end{example}

 While  the relative commutant type is trivially realized by the center of $\CM$,
 the type described in~\eqref{Ex.2} is never realized in $\CM$.  However, the second type is sometimes consistent.
 For instance, if $\CM$ is an infinite dimensional C*-algebras then (2) is consistent.
 Hence not
every consistent type over $\CM$ is necessarily realized in $\CM$.
%but we have the following.

 %\begin{lemma}
 %Every consistent type over $\CM$ is realized in an ultrapower of $\CM$
 %associated to an ultrafilter over a sufficiently large index set.
%\end{lemma}

%\begin{proof}
%For a distinguished variable $x$ let $\Form_{\bfX}[x]$ be the set of
%all $\Lang_{\bfX}$ formulas with at most variable $x$ free. Let $I$
%be the set of all  finite subsets
 %of $\Form_{\bfX}[x]$. Let $\cU$ be a nonprincipal
%ultrafilter on this set. For every $F\in I$ fix $b_F$ in
%$D(p)^{\CM}$ that  satisfies $d(\psi^{\CM}(b), p(\psi(x))) <\e$ for
%all $\psi(x)$ in $F$. Then Proposition~\ref{P.Los}
% implies
% $(b_F: F\in I)$ represents an element of the ultrapower that realizes $p$.
%\end{proof}

 \subsection{Saturation}  A model  $\CM$ of language $\CL$ is \emph{countably saturated}
if for every  countable subset $X$ of the universe of $\CM$,
every consistent type over $X$ is realized in $\CM$.
  More generally, if $\kappa$ is a cardinal then $\CM$ is \emph{$\kappa$-saturated} if
 for every subset $X$ of the universe of $\CM$ of
 cardinality $<\kappa$ every consistent type over $X$ is realized in $\CM$. We say that $\CM$ is \emph{saturated} if it is $\kappa$-saturated where $\kappa$ is the  density character of $\CM$.

 Thus countably saturated is the same as $\aleph_1$-saturated, where $\aleph_1$ is the least uncountable cardinal. The following is a version of a classical theorem of Keisler for the logic of metric structures.
 %Every operator algebraist has seen a proof of some instance of the following classical result (see \cite[Proposition~7.6]{BYBHU}).  [DAVID: I'D OMIT THIS LAST SENTENCE.]

 \begin{proposition} \label{P.ctbly.saturated} If $\CM_i$,
 for $i\in \bbN$,   are  models  of the same language
 and $\cU$ is a nonprincipal ultrafilter on $\bbN$ then the ultraproduct $\prod_{\cU} \CM_i$ is countably saturated.
  If $\CM$ is  separable then the relative commutant  of $\CM$ in $\CM^{\cU}$ is countably saturated.
 \end{proposition}

\begin{proof} A straightforward diagonalization argument, cf. the proof of
Proposition~\ref{L.realizing-type}.
\end{proof}

The following lemma is a key tool.

\begin{lemma} \label{L.embedding}
Assume  $\CN$ is a  countably saturated $\CL$-structure, $\CA$ and $\CB$
are separable $\CL$-structures, and $\CB$ is an elementary submodel
 of $\CA$.
Also assume   $\Psi\colon \CB\to \CN$
 is an  elementary embedding.
Then $\Psi$ can be extended to an elementary embedding $\Phi\colon
\CA\to \CN$. { \spreaddiagramrows{1.5pc}
\spreaddiagramcolumns{1.5pc}
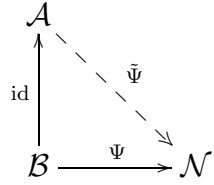
\[
\diagram
 \CA\xdashed[dr]^{\tilde\Psi}|>\Tip &  \\
\CB\rto^{\Psi}\uto^{\id} & \CN
\enddiagram
\]}
\end{lemma}

%\xdashed[dddr]<2ex>|\RNP|>\Tip\\

\begin{proof} Enumerate a countable dense subset of $A$ as $a_n$, for $n\in \bbN$, and
fix a countable dense $B_0\subseteq B$.
Let $t_n$ be the type of~$a_n$ over $B_0\cup \{a_j\colon j<n\}$. If
$t$ is a type over a subset $X$ of $A$ then by $\Psi(t)$ we denote
the type over the $\Psi$-image of $X$ obtained from $t$ by replacing
each $a\in A$ by $\Psi(a)$.
 By countable saturation  realize $\Psi(t_0)$ in $\CN$ and denote the realization
 by $\Psi(a_0)$ in order to simplify the notation.
 The type $\Psi(t_1)$ is realized in $\CN$ by an element that we denote by $\Psi(a_1)$.
 Continuing in this manner, we find elements $\Psi(a_n)$ in $\CN$, for $n\in \bbN$.
 Since the sequence~$a_n$, for $n\in \bbN$, is dense in $A$,
 by elementarity the map $a_n\mapsto \Psi(a_n)$ can be extended to an
 elementary embedding $\Phi\colon \CA\to \CN$ as required.
 \end{proof}

Note that the analogue of Lemma~\ref{L.embedding} holds when,
instead of assuming $\CA$ to be separable, ~$\CN$ is assumed to be
$\kappa$-saturated for some cardinal $\kappa$ greater than the
 density character of $\CA$. Using  a transfinite extension of
Cantor's back-and-forth method,
 Proposition~\ref{P.LST}
 and this analogue of Lemma~\ref{L.embedding} one proves the following.

 \begin{proposition} \label{P.saturated} Assume $\Lang$ is a separable language.
 If $\CM$ and $\CN$ are elementarily equivalent saturated
 models of $\Lang$ that have
 the same uncountable   density  character then they are isomorphic. \qed
 \end{proposition}

% \begin{proof} The proof is a transfinite back-and-forth construction of length $\kappa$, where $\kappa$ is the character density of $M$.
% We find $a_i$, $i<\kappa$, in $\ubM$ and $b_i$, $i<\kappa$, in $\ubN$, such
 %that for every $i$ the type of $a_i$ over $\{a_j: j<i\}$ in $M$ is equal to the type of
 %$\{b_j: j<i\}$ in $N$ (with the identification of $c_{a_j}$ with $a_{b_j}$ for all $j$).
 %Then the map sending $a_i$ to $b_i$ has the unique continuous linear extension to an isomorphism between $M$ and $N$.
% \end{proof}

%%%%%%%

For simplicity, in the following discussion we refer to tracial von
Neumann algebras
 (C*-algebras, unitary
groups of a tracial von Neumann algebra or a C*-algebra,
respectively)  as `algebras.'

\begin{corollary} \label{C.CH}\label{C.CH.C*}
Assume the Continuum Hypothesis. If $M$ is an algebra   of 
density character $\leq \fc$ then all of its ultrapowers associated with
nonprincipal ultrafilters are isomorphic. If $M$ is separable, then
all of its relative commutants in ultrapowers associated
 with nonprincipal ultrafilters are isomorphic.
\end{corollary}

\begin{proof}
The Continuum Hypothesis implies that  such ultrapowers are
saturated and by Proposition~\ref{P.Los},
Proposition~\ref{P.ctbly.saturated}
 and Proposition~\ref{P.saturated}  they are  all isomorphic.
If $M$ is separable, then the isomorphism between the ultrapowers
can be chosen to send the diagonal copy of~$M$ in one ultrapower to
the diagonal copy of~$M$ in the other ultrapower and therefore the
relative commutants are isomorphic.
\end{proof}

 It should be noted that, even in the case when the Continuum Hypothesis fails,
 countable saturation and a transfinite back-and-forth construction together  show
 that ultrapowers of a fixed algebra  are very similar to each other.
  % IN THE SENSE THAT THEY HAVE THE SAME SEPARABLE SUBMODELS(?). IN PARTICULAR (IF).

\begin{corollary}\label{C.Similar}
 Assume $M$ is a separable algebra and $\cU$
and $\cV$ are nonprincipal ultrafilters on $\bbN$. Then for all
separable algebras $N$ we have the following:
\begin{enumerate}
\item $N$ is a subalgebra of $M^{\cU}$ if and only if it is a subalgebra of $M^{\cV}$;
\item $N$ is a subalgebra of $M'\cap M^{\cU}$ if and only if $N$ is a subalgebra of $M'\cap M^{\cV}$.
\end{enumerate}
\end{corollary}

\begin{proof} These classes of algebras are axiomatizable, so instead of algebras we can
 work with the associated models.  Supposing that $\CN \subset \CM^{\cU}$, apply the downward L\"owenheim--Skolem theorem (Proposition~\ref{P.LST}) to find an elementary submodel $\CP$ of $\CM^{\cU}$ whose universe contains $N$ and the diagonal copy of $M$.  Now consider the elementary inclusion $\CM \subseteq \CP$, and use Lemma~\ref{L.embedding} to extend the map which identifies $\CM$ with the diagonal subalgebra of $\CM^{\cV}$, the latter being countably saturated by Proposition~\ref{P.ctbly.saturated}.  This extension carries $P$ onto a subalgebra of $M^{\cV}$ and restricts to an isomorphism from $N$ onto its image.  In case $M$ and $N$ commute, their images in $M^{\cV}$ do too.
\end{proof}

We also record a refining of the fact that the relative commutants
of a separable algebra are isomorphic assuming CH.

\begin{corollary} \label{C.Relative}
Assume $M$, $\cV$ and $\cU$ are as in Corollary~\ref{C.Similar}.
Then the relative commutants $M'\cap M^{\cU}$ and $M'\cap M^{\cV}$
are elementarily equivalent.
\end{corollary}

\begin{proof} By countable saturation of
ultrapowers, a type $p$ over the copy of $M$ inside $M^{\cU}$ is
realized if and only the same type over the copy of $M$ inside
$M^{\cV}$ is realized. By considering only types $p$ that extend the
relative commutant type the conclusion follows.
\end{proof}

The conclusion of Corollary~\ref{C.Relative} fails when $M$ is the C*-algebra
$\cB(\ell^2)$. By \cite{FaPhiSte:Relative} CH implies $\cB(\ell^2)'\cap
\cB(\ell^2)^{\cU}$ is trivial for one $\cU$ and infinite-dimensional for
another. This implies that the assumption of separability is
necessary in Corollary~\ref{C.Relative}.

\section{Stability, the order property, and nonisomorphic ultrapowers}
\label{S.Model-theoretic}

This section defines the two main model theoretic notions of the paper: stability and the order property.  We show that each is equivalent to the negation of the other (Theorem~\ref{T.OP}), and that the order property is equivalent to the existence of nonisomorphic ultrapowers when the continuum hypothesis fails (Theorem \ref{T0}).
While the analogue of the
former fact is well-known in the discrete case, we could not find a reference
to the analogue of the latter fact in the discrete case.
We have already seen that when the continuum hypothesis holds all ultrapowers are isomorphic (Corollary \ref{C.CH}).  (Of course we are talking about separable structures with ultrapowers based on free ultrafilters of $\mathbb{N}$.)

\subsection{Stability}\label{S.Stability}
%In the present section we assume familiarity with model theory of metric structures (\cite{BYBHU}).
%Some of the background is given in the appendix, \S\ref{S.Logic}.
%In the previous sections we have shown that the isomorphism of  ultrapowers of II$_1$ factors, of C*-algebras, and of their unitary groups  is sensitive to set-theoretic axioms.
%In Theorem~\ref{T0} below we show that the metric structures with this property are exactly the ones whose theory in the logic of metric structures is unstable.
%Note that the stability of Banach spaces is well-studied (see \cite{Iovino:StableI}).

%For  a subset $X$ of a model $A$ by $S(X)$  we denote the set of all types
%over~$X$ (see \S\ref{S.Types}).  As is common in stability theory, in this section we write $AB$ for
%$A\cup B$.

\begin{definition}
We say a theory $T$
 is \textbf{$\lambda$-stable} if for any model $M$ of $T$
 of density character $\lambda$, the space of complete types $S(M)$ has density character $\lambda$ in the metric topology on $S(M)$.  We say $T$ is
 \textbf{stable} if it is stable for some $\lambda$ and $T$ is \textbf{unstable} if it is not stable.
 \end{definition}

For a theory $T$ in a separable language one can show that $T$ is stable if and only if
 it is $\fc$-stable (see the proof of Theorem~\ref{T.OP}).

Our use of the term ``stable" in this paper agrees with model theoretic terminology
in both continuous and discrete logic.  Motivated by model theory, in 1981 Krivine and Maurey defined
a related notion of stability for Banach spaces that is now more familiar to many analysts (\cite{KM}).  It is characterized by the requirement
\begin{equation}
\tag{*} \lim_{i \to \CU} \lim_{j \to \CV} \|x_i + y_j\| = \lim_{j \to \CV} \lim_{i \to \CU} \|x_i + y_j\|
\end{equation}
for any uniformly bounded sequences $\{x_i\}$ and $\{y_j\}$, and any
free ultrafilters $\CU, \CV$ on $\mathbb{N}$. One can show
(\cite{Iovino:StableI}) that a Banach space satisfies (*)
if and only if no quantifier-free formula has the order property in that structure(cf. \S \ref{S.OP}), so model theoretic stability of the theory of a Banach space $X$ implies stability of $X$in the sense of Krivine-Maurey.

We proved in \cite{FaHaSh:Model1} that all infinite-dimensional C*-algebras are unstable.
The same cannot be said for infinite-dimensional Banach algebras: take a stable Banach space
and put the zero product on it. However a stable Banach space can become unstable when it is turned
into a Banach algebra.  We exhibit this behavior in Proposition \ref{T.BAlg} below.

\subsection{The order property} \label{S.OP}

\begin{definition}\label{Def.OP}
We say that a continuous theory $T$ \textbf{has the order property}
if \begin{itemize}
\item there is a formula $\psi(\bar x,\bar y)$ with the lengths of $\bar x$ and $\bar y$ the same, and a
sequence of domains $\bar D$ consistent with the sorts of $\bar x$ and $\bar y$, and
\item a model $M$ of $T$ and $\langle \bar a_i : i \in \bbN \rangle \subseteq \bar D(M)$
\end{itemize}
such that
$$\psi(a_i,a_j) = 0 \mbox{ if $i < j$ and } \psi(a_i,a_j) = 1 \mbox{ if }i \geq
j.$$

Note that these evaluations are taking place in $M$. Also note that
by the uniform continuity of $\psi$, there is some $\e>0$ such that
$d(\bar a_i,\bar a_j) \geq \e$ for every $i \neq j$ where the metric
here is interpreted as the supremum of the coordinatewise metrics.
\end{definition}

\begin{proposition}
$\Th(A)$ has the order property if and only if there is $\psi$ and
$\bar D$ such that for all $n$ and $\delta > 0$, there are
$a_1,\ldots,a_n \in \bar D(A)$ such that
%\item $\varphi_n(a_i) \leq \delta$ for all $i$,
%\item $d(a_i,a_j) \geq \e$ for $i \neq j$, and
$$\psi(a_i,a_j) \leq \delta \mbox{ if $i < j$ and } \psi(a_i,a_j) \geq 1 -
\delta \mbox{ if } i \geq j.$$

\end{proposition}

\begin{proof} Compactness. \end{proof}

\begin{definition}
Suppose that $M$ is a metric structure and $p(\bar x) \in S^{\bar
D}(M)$ is a type.  We say that $p$ is \textbf{finitely determined} if for
every formula $\varphi(\bar x,\bar y)$, choice of domains $\bar D'$
consistent with the variables $\bar y$, and $m \in \bbN$, there is $k
\in \bbN$ and a finite set $B \subseteq \bar D(M)$ such that for
every $\bar c_1,\bar c_2 \in \bar D'(M)$ (see Notation~\ref{Not.type})
$$ \sup_{\bar b \in B} |\varphi(\bar b,\bar c_1) - \varphi(\bar b,\bar c_2)| \leq \frac{1}{k} \quad \Rightarrow \quad
|\varphi(p,\bar c_1) - \varphi(p,\bar c_2)| \leq \frac{1}{m}.$$
\end{definition}

%All instances of the order property mentioned earlier in the paper are examples of the order property in the associated metric structure.  Specifically,
%\begin{enumerate}
%\item By Lemma \ref{L.A1}, $\varphi(x_1,x_2,y_1,y_2) =||[x_1,y_2]||_2$ shows that the theory of any II$_1$ factor has the order property.  By Lemma \ref{L.A1.RC}, the same formula (but with a little more work) shows that the relative commutant of any separable II$_1$ factor has the order property;
%\item MISSING HERE something for unitary groups of II$_1$ factors and their relative commutants (proved indirectly in the proof of Theorem~\ref{T.type-II-1});
%\item something for relative commutant of unitary groups;
%\item By Lemma \ref{L.B3}, the formula $\varphi(x,y) = ||xy - y||$ shows that any infinite-dimensional C*-algebra has the order property; again, the same formula but with more work shows that the relative commutant of any separable C*-algebra has the order property;
%\item By Lemma \ref{unitary-lemma}, $\varphi(x_1,x_2,y_1,y_2) = ||1 - x_1y_2||$ shows that the theory of the unitary group of any infinite-dimensional unital C*-algebra, as well as the theory of its relative commutant,  has the order property.
%\end{enumerate}

\begin{theorem}\label{T.OP}
The following are equivalent for a continuous theory $T$:
\begin{enumerate}
\item $T$ is stable.
\item $T$ does not have the order property.
\item If $M$ is a model of $T$ and $p \in S(M)$ then $p$ is finitely determined.
\end{enumerate}
\end{theorem}

\begin{proof}
(1) implies (2) is standard: suppose that $T$ has the order property
via a formula $\theta$ and choose any cardinal $\lambda$.  Fix $\mu
\leq \lambda$ least such that $2^\mu > \lambda$ (note that $2^{<\mu}
\leq \lambda$).  By compactness, using the order property, we can
find $\langle \bar a_i : i \in 2^{<\mu} \rangle$ such that
$\theta(\bar a_i,\bar a_j) = 0$ if $i < j$ in the standard
lexicographic order and 1 otherwise.  Clearly, $\chi(S(A)) >
\chi(A)$ where $A = \{\bar a_i : i \in 2^{< \mu} \}$ so $T$ is not
$\lambda$-stable for any $\lambda$.

To see that (3) implies (1), fix a model $M$ of $T$ with density
character $\lambda$ where $\lambda^{\chi(T)} = \lambda$.  By
assumption, every type over $M$ is finitely determined and so there
are at most $\lambda^{\chi(T)} = \lambda$ many types over $M$ and so
$T$ is $\lambda$-stable.

Finally, to show that (2) implies (3), suppose that there is a type
over a model of $T$ which is not finitely determined. Fix $p(\bar x)
\in S^{\bar D}(M)$, $\varphi(\bar x,\bar y)$, domains $\bar D'$
consistent with the variables $\bar y$ and $m \in \bbN$ so that
% $d(x,c) \geq \e \in q$ for all $c \in M$, and
for all $k$ and finite $B \subseteq \bar D(M)$, there are $n_1,n_2
\in \bar D'(M)$ such that
 \[\max_{b \in B}|\varphi(b,n_1) - \varphi(b,n_2)| \leq
1/k\] but
 \[|\varphi(p,n_1) - \varphi(p,n_2)| > 1/m.\]

 We now use this $p$ to construct an ordered sequence.  Define
 sequences $a_j,b_j,c_j$ and sets $B_j$ as follows:  $B_0 = \emptyset$.
 If we have defined $B_j$, choose $b_j,c_j \subseteq \bar D'(M)$ such that
  $\max_{b \in B_j}|\varphi(b,b_j) - \varphi(b,c_j)|
\leq 1/2m$ but
 $|\varphi(p,b_j) - \varphi(p,c_j)| > 1/m$.

Now choose $a_j \in
 \bar D(M)$ so that $a_j$ realizes $\varphi(\bar x,b_i) = \varphi(p,b_i)$ and
 $\varphi(\bar x,c_i) = \varphi(p,c_i)$ for all $i \leq j$.  Let $B_{j+1} =
 B_j\cup\{a_j,b_j,c_j\}$.  It follows that if $i \geq j$ then
 $|\varphi(a_i,b_j) - \varphi(a_i,c_j)| > 1/m$.  If $i < j$ then
 $|\varphi(a_i,b_j) - \varphi(a_i,c_j)| \leq 1/2m$ since $a_i \in
 B_j$.  Consider the formula
 \[\theta(x_1,y_1,z_1,x_2,y_2,z_2) := |\varphi(x_1,y_2) -
 \varphi(x_1,z_2)|.\]
Then $\theta$ orders $\langle a_i,b_i,c_i : i \in \bbN \rangle$.
\end{proof}

%\subsection{Nonisomorphic ultrapowers}

\begin{theorem}\label{T0}
Suppose that $A$ is a separable metric structure in a separable
language.
\begin{enumerate}
\item\label{T0.1}
If the theory of $A$ is stable then for any two non-principal
ultrafilters $\cU,\cV$ on $\bbN$, $A^{\cU} \cong A^{\cV}$.
\item \label{T0.2} If the theory of $A$ is unstable then the following are equivalent:
\begin{enumerate}
\item $A$ has  fewer than $2^{2^{\aleph_0}}$ nonisomorphic ultrapowers associated with
nonprincipal ultrafilters on $\bbN$.
\item for any two non-principal ultrafilters $\cU,\cV$ on
$\bbN$, $A^{\cU} \cong A^{\cV}$;
\item the Continuum Hypothesis holds.
\end{enumerate}
\pushcounter
\end{enumerate}
\end{theorem}

It is worth mentioning that Theorem~\ref{T0} is true in the first order context,
as can be seen by considering a model of a first-order theory as a metric model with respect to the
discrete metric.  Although this is undoubtedly known to many, we were unable to find a direct reference. The proof of (1) will use tools from stability theory, and the reader
may want to refer to \cite{BYU:ContStab} for background.

\begin{proof} \eqref{T0.1} Assume that the theory of $A$ is stable. We will show that $
A^{\cU}$ is $\fc$-saturated and so it will follow that $A^{\cU} \cong
A^{\cV}$ no matter what the size of the continuum is (see Proposition~\ref{P.saturated}).

So suppose that $B \subseteq A^{\cU}$,  $|B| < \fc$, and $q$ is a
type over $B$. We may assume  that $B$ is an elementary submodel and
that $q$ is nonprincipal and complete. As the theory of $A$ is
stable, choose a countable elementary submodel $B_0 \subseteq B$ so
that $q$ does not fork over $B_0$. We shall show that in $A^{\cU}$
one can always find a Morley sequence in $q|_{B_0}$ of size $\fc$.

Towards this end, 
fix a countable Morley sequence $I$ in the type of $q|B_0$ and let $\bar q = tp(I/B_0)$, a type in the variables $x_n$ for $n \in \bbN$.  Since $B_0$ is countable and the language is separable, there are countably many formulas $\psi_n(x_1,\ldots,x_n,b_n)$ over $B_0$ such that $\psi_n(x_1,\ldots,x_n,b_n) = 0 \in \bar q$ and $\{\psi_n(x_1,\ldots,x_n,b_n) = 0 : n \in \bbN\}$ axiomatizes $\bar q$.

%suppose that $q|_{B_0} = \{\varphi_n(x) = 0 : n
%\in \bbN\}$ and fix a countable Morley sequence $I$ in the the type
%$q |_{B_0}$. Fix formulas $\psi_n(x_1,\ldots,x_n)$ over $B_0$ which
%have the properties that
%\begin{enumerate}
%\popcounter
%\item[(i)] $\{\psi_n(x_1,\ldots,x_n) = 0: n \in \bbN\}$ implies that
%the $x_i$'s are independent over $B_0$ (realize the type of $I$ over
%$B_0$);
%\item[(ii)] each $x_i$ realizes $q |_{B_0}$; more
%specifically, $\psi_n(x_1,\ldots,x_i,\ldots,x_n) \geq
%\varphi_n(x_i)$ for all choices of $x_1,\ldots,x_n$ (these
%evaluations are taking place in $\prod_U A$).
%\end{enumerate}

Now let $D_i = \{ n \geq i : \inf_x \psi_i(x,b_i(n)) < 1/i\}$
%, where
%we make the parameters in $\psi_i$ explicit. 
For a fixed $n$,
consider $\{i : n \in D_i\}$. This set has a maximum element; call it
$i_n$. Now fix $a^n_1,\ldots,a^n_{i_n} \in A$ such that
$\psi_{i_n}(a^n_1,\ldots,a^n_{i_n},b_{i_n}) <1/i_n$.  Now consider the set $J$ of
all $g:\bbN \rightarrow A$ such that $g(n) \in
\{a^n_1,\ldots,a^n_{i_n}\}$ for all $n$.  Any $g \in J$ will satisfy $q|B_0$ in
$A^{\cU}$ since every element of $I$ realized that type.  If $g_0,\ldots,g_k$ are in $J$ and distinct modulo $U$ then they are
independent over $B_0$ since $I$ was a Morley sequence.  To finish then, we need to see that there are $\fc$-many
distinct $g$'s modulo $U$. This follows from the fact that the
$i_n$'s are not bounded. To see this, for a fixed $m$, let $X = \{ n
\geq m : \inf_x\psi_m(x,b(n)) \leq 1/m \}$. Pick any $n \in X$. We
have that $n \in D_m$ so $i_n \geq m$ and we conclude that the
$i_n$'s are not bounded.

Since $|B| < \fc$, there is a $J_0$ of cardinality less than
$\fc$ such that $B$ is independent from $J$ over $J_0$.  Choosing
any $a \in J \setminus J_0$ and using symmetry of non-forking
remembering that $J$ is a Morley sequence over $B_0$, it follows
that $a$ is independent from $B$ over $B_0$.  Since $B_0$ is a
model, $q|_{B_0}$ has a unique non-forking extension to $B$ and it
follows that $a$ realizes $q$.  This finishes the proof of~\eqref{T0.1}.

%Since $|B| < \fc$, there is a $D_0$ of cardinality less than $\fc$ such that $B$ is independent from $D$ over $D_0$.  Choosing any $a \in D \setminus D_0$ and using symmetry of non-forking remembering that $D$ is a Morley sequence over $B_0$, that $a$ is independent from $B$ over $B_0$.  It follows that $a$ realizes $q$.

\eqref{T0.2}
 If the Continuum Hypothesis  holds then $A^{\cU}$ is always saturated and so
for any two ultrafilters $\cU,\cV$, $A^{\cU} \cong A^{\cV}$ even if
we fix the embedded copy of $A$ (Corollary~\ref{C.CH}).

The implication (a) implies (c) follows from \cite[Theorem~3]{FaSh:Dichotomy} and of course (b) implies (a).
However, (b) implies (c) also can be proved by  a minor modification of proof from \cite{FaHaSh:Model1}
(which is in turn a modification of a proof from
\cite{Do:Ultrapowers}), so we assume that the reader has a copy of the former handy and
we sketch the differences.
Assume the theory of $A$ is unstable.
Then by Theorem~\ref{T.OP} it has the order property.
The formula $\psi$ witnessing the order property satisfies
\cite[Properties~2.1]{FaHaSh:Model1}
by Corollary~\ref{C.Properties.g}.
Therefore the analogues of
\cite[Lemma~2.4, Lemma~2.5 and Proposition~2.6]{FaHaSh:Model1}
 can be
proved by quoting their proofs verbatim. Hence if   $\cU$ is a
nonprincipal ultrafilter on $\bbN$ then $\kappa(\cU)=\lambda$
(defined in \cite[paragraph before Lemma~2.5]{FaHaSh:Model1})
 if and only if there is a $(\aleph_0,\lambda)$-$\psi$-gap in $A^{\cU}$.

By  \cite[Theorem~2.2]{Do:Ultrapowers}, if CH fails then there are ultrafilters $\cU$ and $\cV$ on
$\bbN$ such that         $\kappa(\cU)\neq \kappa(\cV)$ and this concludes the proof.
\end{proof}

%It was recently proved in \cite{FaSh:Dichotomy} that if the theory of $A$ is unstable and the
%Continuum Hypothesis fails then $A$ has $2^{2^{\aleph_0}}$ nonisomorphic ultrapowers
%associated with nonprincipal ultrafilters on~$\bbN$. Therefore to the list of equivalent statements
%in (2) of Theorem~\ref{T0} one can add the following.
%\begin{enumerate}
%\item [(c)] $A$ has  fewer than $2^{2^{\aleph_0}}$ nonisomorphic ultrapowers associated with
%nonprincipal ultrafilters on $\bbN$.
%\end{enumerate}

\section{Concluding Remarks}
In this final section we include two examples promised earlier and state
three rather different problems. %Two of them  illustrate the relevance of the model-theoretic approach to understanding C*-algebras and~II$_1$ factors, respectively, while the remaining problem is of a rather abstract nature.

\subsection{A non-axiomatizable category of C*-algebras}
Recall that UHF, or uniformly hyperfinite,  algebras are C*-algebras that are C*-tensor products of (finite-dimensional) matrix algebras.
They form a subcategory of C*-algebras and the morphisms between them are
 *-homomorphisms.

\begin{proposition} \label{P.UHF} The category of UHF algebras is not axiomatizable.
\end{proposition}

\begin{proof} By Proposition~\ref{P.ultrapowers} it will suffice to show that this
category  is not closed under taking (C*-algebraic) ultraproducts.  We do this by repeating an argument from Ge-Hadwin (\cite[Corollary 5.5]{GeHa}) exploiting the fact that UHF algebras have unique traces that are automatically faithful.

Let $A$ be the CAR algebra $\bigotimes_{n\in \bbN} M_2(\bbC)$ with trace $tr$, let $\cU$ be
a nonprincipal ultrafilter  on $\bbN$, and let $\{p_n\} \subset A$ be projections with $tr(p_n) = 2^{-n}$.  The sequence $(p_n)$ represents a nonzero projection in $A^\cU$, but $tr^\cU((p_n)) = 0$.  Thus $tr^\cU$ is a non-faithful tracial state, so that $A^\cU$ is not UHF.
%Assume the  ultrapower $A^{\cU}$ is a UHF algebra. Every UHF algebra has a unique trace $tr$ such that $tr(p)>0$ for every nonzero projection $p$ and that the set $\{tr(p)\colon p$ is a projection$\}$ is dense in $[0,1]$ (see e.g., \cite{Black:Operator}). On the other hand, this ultrapower is   countably saturated by Proposition~\ref{P.ctbly.saturated}. By the uniqueness, the trace of $A^{\cU}$ agrees with the trace of $A$ on the diagonal copy of $A$.   Let $p_n$, for $n\in \bbN$, be a sequence of projections in the diagonal copy of $A$ inside $A^{\cU}$   such that $tr(p_n)=2^{-n}$.  Since~$A^{\cU}$ is countably saturated, there is a   nonzero projection $p\in A^{\cU}$ such that $p\leq p_n$ for all $n$. Therefore $tr(p)=0$,   a contradiction.
  \end{proof}
The same argument shows that simple C*-algebras are not axiomatizable.

Since every UHF algebra has a unique trace
one could also consider  tracial ultraproducts, instead of norm-ultraproducts,
   of UHF algebras. However,
  such an ultraproduct is always a II$_1$ factor (\cite[Theorem 4.1]{HL}) and therefore
  not a UHF algebra (because projections in UHF algebras have rational traces).

\subsection{An unstable Banach algebra whose underlying Banach space is stable} \label{S:ba}

The $L^p$ Banach spaces ($1 \leq p < \infty$) are known to be stable (\cite[Section 17]{BYBHU}), and they become stable Banach algebras when endowed with the zero product.  Actually $\ell^p$ ($1 \leq p <\infty$) with pointwise multiplication is also stable; this can be shown by methods similar to \cite[Lemma 4.5 and Proposition 4.6]{FaHaSh:Model1}.
%More interestingly, $L^p$ Banach spaces ($1\leq p\leq \infty$) associated with a probability measure space are stable (albeit not necessarily Banach algebras) when endowed with the pointwise product. More precisely, the categories of such $L^p$ spaces are all equivalent to the category of probability measure algebras.  The case of $L^\infty$ was proved in  \cite[Lemma~4.5]{FaHaSh:Model1}, and the general proof is very similar.
 We now prove that the usual convolution product turns $\ell^1(\bbZ)$ into an unstable Banach algebra, as was mentioned in \S\ref{S.Stability}.

\begin{proposition} \label{T.BAlg}
The Banach algebra $\ell^1 = \ell^1(\bbZ,+)$ (with convolution product) is unstable.
\end{proposition}

\begin{proof}
It suffices to show the order property for $\ell^1$.  This means we must give a formula $\varphi(x,y)$ of two variables (or $n$-tuples) on $\ell^1$, a bounded sequence $\{x_i\} \subset \ell^1$, and $\varepsilon>0$ such that $\varphi(x_i, x_j) \leq \varepsilon$ when $i \leq j$ and $\varphi(x_i, x_j) \geq 2\varepsilon$ when $i > j$.

Let $\{f_n\}_{n \in \bbZ}$ denote the standard basis for $\ell^1$, so that multiplication is governed by the rule $f_m f_n = f_{m+n}$.  Also let $\ell^1 \ni x \mapsto \hat{x} \in C(\mathbb{T})$ be the Gel'fand transform on $\ell^1$, so that $\hat{f}_n$ is the function $[e^{it} \mapsto e^{int}]$.  The Gel'fand transform is always a contractive homomorphism; on $\ell^1$ it is injective but not isometric.

We take $\varphi(x,y) = \inf_{\|z\| \leq 1} \|x z - y\|$, $x_i = (\frac{f_1 + f_{-1}}{2})^{2^i}$, and $\varepsilon = \frac18$.  Note that all the $x_i$ are unit vectors, being convolution powers of a probability measure on $\bbZ$, and $\hat{x}_i = [e^{it} \mapsto (\cos t)^{2^i}] \in C(\mathbb{T})$.

For $i \leq j$, we have $\varphi(x_i, x_j) = 0$ by taking $z = (\frac{f_1 + f_{-1}}{2})^{2^j - 2^i}$.

For $i > j$, let $t_0 \in (0,2\pi)$ be such that $(\cos t_0)^{2^j} = \frac12$.  For any $z \in (\ell^1)_{\leq 1}$,
$$\|x_i z - x_j \|_{\ell^1} \geq \|\hat{x}_i \hat{z} - \hat{x}_j \|_{C(\mathbb{T})} \geq \left |\left(\frac{1}{2}\right)^{2^{i -j}} \hat{z}(e^{it_0}) - \frac12 \right| \geq \frac14,$$
where the middle inequality is justified by evaluation at $t_0$.  We conclude that $\varphi(x_i,x_j) \geq \frac14$ as desired.
\end{proof}

\begin{question}
Is there a nice characterization of stability for Banach algebras?
\end{question}

It is obviously necessary that the underlying Banach space be stable.
% maybe the additional condition is that the norm be submultiplicative in some strong sense.  %We had some other ideas for showing instability with the function $\varphi$ as above, but this required a ``cancellation property" which seemed too good to be true.  Really we didn't think about it enough.  Note that this $\varphi$ is not quantifier-free, so it does not follow that any Banach algebra containing $\ell^1$ is unstable.  (I don't know if there is an infinite-dimensional Banach algebra all of whose abelian subalgebras are finite-dimensional, but it is apparently possible by recent work of Zelazko for an inf.-dim. B. alg. to have a finite-dimensional maximal abelian subalgebra -- which is impossible in a C*-algebra.)
The proof of Proposition \ref{T.BAlg} works for any abelian Banach algebra that contains a unit vector $g$ such that the range of $\hat{g}$ contains nonunit scalars with modulus arbitrarily close to 1.  (Above, $g$ was $\frac{f_1 + f_{-1}}{2}$.)  So for instance the convolution algebra $L^1(\mathbb{R},+)$ is covered, and in fact \textit{any} probability density will do for $g$ in this case.  %This method doesn't work for $\ell^p$ ($1 \leq p < \infty$) with pointwise multiplication.

%\begin{lemma} \label{L.definable.type}
%Assume $M$ is a model with stable theory and $p$ is a type over some
%$B\subseteq M$. Then there is a countable $B_0\subseteq B$ such that
%the restriction of $q$ to $B_0$ is nonforking and moreover for every
%nonforking type $q'$ over $B_0$ such that $q|B_0=q'|B_0$ we have that
%$q=q'$.\marginpar{Proof using models!}
%\end{lemma}

%\begin{proof} By \cite[Theorem~14.16]{BYBHU} type $p$
%is definable. If $B_0$ is the set of all parameters occurring in the formulas that
%define $p$, then $B_0$ is countable and it is as required.
%\end{proof}

%\subsection{The number of nonisomorphic ultrapowers}
%
%ILIJAS - IS THIS GOING?
%
%
%Our Theorem~\ref{T0} uses a result of \cite{Do:Ultrapowers} and therefore
%provides
% only as many nonisomorphic ultrapowers of a separable model of an unstable
% theory as there are uncountable cardinals
%less than $\fc$ (i.e., two).
%In \cite{KrShTeTh} it was proved that if the Continuum Hypothesis
%fails then there are $2^{2^{\aleph_0}}$ nonisomorphic ultrapowers of $(\bbN, \leq)$ associated
%to nonprincipal ultrafilters on $\bbN$, and therefore a positive answer to
%the following question would significantly improve this lower bound.
%
%
%
%\begin{question} \label{Q6.2} Assume $A$ and $B$ are separable models of an unstable
%theory, and $\cU$ and $\cV$ are nonprincipal ultrafilters on $\bbN$.
%If $A^{\cU}$ and $A^{\cV}$ are isomorphic, are $B^{\cU}$ and $B^{\cV}$ necessarily
%isomorphic?
%\end{question}
%
%
%We do not know the answer to the analogous question in the first order context.
%
%
%

\subsection{K-theory reversing automorphisms of the Calkin algebra}
A well-known problem of Brown--Douglas--Fillmore (\cite[1.6(ii)]{BrDoFi:Extensions}) asks
whether there is an automorphism
of the Calkin algebra that sends the image of the unilateral shift to its adjoint.
The main result of~\cite{Fa:All} implies that if ZFC is consistent then there is a model of ZFC in
which there is no such automorphism. A deep metamathematical result
of Woodin, known as the \emph{$\Sigma^2_1$-absoluteness theorem}, essentially (but not literally)
implies that the Brown--Douglas--Fillmore question has a positive answer if and only
if the Continuum Hypothesis implies a positive answer (see~\cite{Wo:Beyond}).
The type referred to in the following question is the type over the empty set in the sense of
\S\ref{S.Types}.

\begin{question} \label{Q.BDF} Do the image of the unilateral shift in the Calkin algebra and its
adjoint have the same type in the Calkin algebra?
\end{question}

A negative answer to Question~\ref{Q.BDF} would imply a negative answer to the Brown--Douglas--Fillmore problem. A positive answer to the former would reduce the latter to a problem about
model-theoretic properties of the Calkin algebra. The Calkin algebra is not
countably saturated (first author, unpublished),  but it may be countably homogeneous.  For example, its poset of projections is countably saturated (\cite{Had:Maximal}, see also
\cite{Pede:Corona}).
Countable homogeneity of the Calkin algebra
 would, in the presence of the Continuum Hypothesis, imply that Question~\ref{Q.BDF}
 has a positive answer if and only if the Brown--Douglas--Fillmore question has a positive answer.
It is not difficult to prove that
the (model-theoretic) types of operators in the Calkin algebra, and therefore the
answer to Question~\ref{Q.BDF}, are
absolute between transitive models of ZFC.

\subsection{Matrix algebras}
We end with discussion of finite-dimensional matrix algebras and
a result that partially complements  \cite[Proposition~3.3]{FaHaSh:Model1},
where it was proved that if the Continuum Hypothesis fails then the matrix algebras $M_n(\bbC)$,
for $n\in \bbN$, have nonisomorphic tracial ultraproducts.

\begin{prop} Every increasing sequence $n(i)$, for $i\in \bbN$, of
natural numbers has a further subsequence $m(i)$, for $i\in \bbN$
such that if the Continuum Hypothesis holds then all tracial ultraproducts of
$M_{m(i)}(\bbC)$, for $i\in \bbN$, are isomorphic.
\end{prop}

\begin{proof}
The set of all $\Lang$-sentences (see \S\ref{S.Syntax}) is
separable. Let $\bT_n=\Th(M_n(\bbC))$, the map associating the value
$\psi^{M_n(\bbC)}$ of a sentence $\psi$ in $M_n(\bbC)$ to $\psi$.
Since the set of sentences is separable we can pick a sequence
$m(i)$ so that the theories $\bT_{m(i)}$ converge pointwise to some
theory $\bT_\infty$. Let $\cU$ be a nonprincipal  ultrafilter such
that $\{m(i): i\in \bbN\}\in \cU$. By \L os's theorem,
Proposition~\ref{P.Los}, $\Th(\prod_{\cU} M_n(\bbC))=\bT_\infty$. By
Proposition~\ref{P.ctbly.saturated} and the Continuum Hypothesis all
such ultrapowers are saturated and therefore
Proposition~\ref{P.saturated} implies all such ultrapowers are
isomorphic.
 \end{proof}

\begin{question} \label{Q.M-n} Let $\psi$ be an $\Lang$-sentence.
Does $\lim_{n\to \infty} \psi^{M_n(\bbC)}$ exist?
\end{question}

A positive answer is equivalent to the assertion
that all ultraproducts $\prod_{\cU} M_n(\bbC)$ are elementarily
equivalent, and therefore isomorphic if the Continuum Hypothesis is assumed (see Proposition~\ref{P.ctbly.saturated} and
Proposition~\ref{P.saturated}).

\bibliographystyle{amsplain}
\bibliography{stablebib}
\end{document}